\documentclass[12pt]{amsart}
\usepackage{amsmath,amsthm, amssymb}
\input epsf

\newtheorem{theorem}{Theorem}

\newtheorem{corollary}{Corollary}

\newtheorem{remark}{Remark}
\newtheorem{observation}{Observation}

\def\IR {\text{\it IR}}
\def\a{{\buildrel {\text{ind}} \over \longrightarrow}}
\def\na{{\buildrel {\text{ind}} \over \nrightarrow}}
\def\rem{\mathrm{rem}}
\begin{document}
%--------------------------------------------------------------
\title[ ON INDUCED RAMSEY NUMBERS FOR MULTIPLE COPIES OF GRAPHS]
{ON INDUCED RAMSEY NUMBERS FOR MULTIPLE COPIES OF GRAPHS}
%\date {7 November 2017}%--------------------------------------------------------------
\author{Maria Axenovich}
\address{Department of  Mathematics\\
   Karlsruher Institute of Technology\\
   Englerstr. 2, D-76131 Karlsruhe \\
   Germany}
\email{maria.aksenovich@kit.edu}

%--------------------------------------------------------------
\author{Izolda Gorgol}
\address{Department of Applied Mathematics\\
   Lublin University of Technology\\
    ul.\ Nadbystrzycka 38D, 20-618 Lublin\\
    Poland}
\email{i.gorgol@pollub.pl}
%--------------------------------------------------------------
%-----------------------------------------------------
\begin{abstract}
We say that a graph $F$ {\it strongly arrows} a pair of graphs $(G,H)$ and write $F \a (G,H)$ if  any colouring of its edges with red and blue  
leads to  either a red $G$  or a blue $H$ appearing as induced subgraphs of $F$. 
{\it The induced Ramsey number}, $\IR(G,H)$ is defined as $\min\{|V(F)|: F\a (G,H)\}$.
We consider the connection between the induced Ramsey number for a pair of two connected graphs $\IR(G,H)$ and the induced Ramsey number for multiple copies of these graphs $\IR(sG,tH)$, where   $xG$ denotes the  pairwise vertex-disjoint union of $x$ copies of $G$. It is easy to see that  if  $F\a (G,H)$ then $(s+t-1)F \a (sG, tH)$. 
This implies that 
$$\IR(sG, tH) \leq (s+t-1)\IR(G,H).$$
For all known results on induced Ramsey numbers for multiple copies, the inequality above holds as equality.
We show that there are infinite classes of  graphs for which  the inequality above is strict and moreover, $\IR(sG, tH)$ could be arbitrarily smaller than 
$(s+t-1)\IR(G,H)$. 
On the other hand, we provide further examples of classes of graphs for which the inequality above holds as equality.
\end{abstract}
%-----------------------------------------------------
\maketitle
%----------------------------------------------------

\section{Introduction}

We say that a graph $F$ {\it strongly arrows} a pair of graphs $(G,H)$ and write $F\a (G,H)$ if  any colouring of its edges with red and blue  
leads to  either a red copy of $G$  or a blue copy of $H$ appearing as induced subgraphs of $F$. We call the graph $F$ {\it strongly arrowing graph}. 
Here, a  graph $G$ is an {\it induced subgraph} of a graph $F$, denoted by $G\prec F$,  if $G$ is a subgraph of $F$ and
two vertices of $G$ form an edge in $G$ if and only if they form an edge in $F$.  
A {\it copy of a graph} is an isomorphic image of the graph. When it is clear from context, we simply write $F$ instead of a copy of $F$.
For graphs $G$ and $H$,  the {\it induced Ramsey number}, $\IR(G,H)$ is defined as $\min\{|V(F)|: F\a (G,H)\}$. It is a generalization of  classical  Ramsey numbers $R(G,H)$, where we color the edges of a complete graph and do not require the monochromatic copies to be induced, i.e., 
$R(G, H)$ is the smallest integer $n$ such that any edge coloring of $K_n$, a complete graph on $n$ vertices, in red and blue contains either a red copy of $G$ and a blue copy of $H$ as a subgraph. 
It is a corollary of the famous theorem of Ramsey that these numbers are always finite.\\

The existence of the induced Ramsey numbers is not obvious and it was a subject of an intensive study. Finally it was
 proved independently by Deuber \cite{deuber}, Erd\H os, Hajnal, and P\'osa \cite{erdhapo}, and R\"odl \cite{rodl,rodlp}. Since in case of complete graphs the induced subgraph is the same as the subgraph it is obvious that $\IR(K_m,K_n)=R(K_m,K_n)$. When at least one of the graphs in the pair is not complete these functions may differ.
  \\

The known results on induced Ramsey numbers are mostly of asymptotic type and mostly  concern the upper bounds.
Erd\H os conjectured \cite{erdirconj} that there is a positive constant $c$ such  that every graph $G$ with $n$ vertices satisfies $\IR(G,G)\le 2^{cn}$. The most recent result in that direction is that of Conlon, Fox, and Sudakov \cite{confoxsud} who showed that $\IR(G,G)\le 2^{cn\log n}$ improving the earlier result of  Kohayakawa, Pr\"omel and R\"odl  \cite{kohprro} 
that stated that $\IR(G,G)\le 2^{cn(\log n)^2}$. The known  upper bounds are obtained either by probabilistic (\cite{becksize1, haxkohlu, kohprro, luczdeg}) or by constructive methods \cite{schaefer}.  A comparision of results of both types can be found in the paper of Shaefer and Shah \cite{schaefer}. 
%The authors of  \cite{schaefer} give explicit constructions of arrowing graphs for a number of pairs of graphs including trees, complete graphs, bipartite graphs, and cycles.  
Fox and Sudakov in \cite{fox_sudakov_induced} present a  unified  approach  to  proving  Ramsey-type  theorems  for  graphs  with  a  forbidden induced  subgraph which can be used  in finding   explicit  constructions  for  upper  bounds  on  various  induced  Ramsey numbers.  

Simple lower bounds on induced Ramsey numbers follow from classical Ramsey numbers: $\IR(G,H) \ge R(G,H).$
Another general approach for the lower bounds on $\IR(G,H)$, where $H$ has chromatic number $k$ is to partition the vertex set of a given graph in $k-1$ parts such that each part does not induce $G$, color all edges within the parts red and all edges between the parts blue. The number of such parts is dictated by generalized chromatic numbers, see among others a paper by Albertson, Jamison, Hedetniemi,  and Locke \cite{Albertson_subchromatic} and the references therein. %generalized_Bollobas,generalized_Broere}. 

Unfortunately these lower bounds are not strong enough and the best lower bounds are obtained by a  careful structural analysis of a given graph.
Recently a step in that direction was done by Gorgol in \cite{gorgol-LB}. She showed that the lower bound for the induced Ramsey number for a connected graph $G$ with an independence number $\alpha$ versus a graph $H$ with the clique number $\omega$ can be expressed as $(\alpha-1)\frac{\omega(\omega-1)}2+\omega$. 
\\

 We focus on induced Ramsey numbers for multiple copies of connected graphs. 
The ordinary Ramsey numbers for multiple copies of connected graphs were considered by Burr, Erd\H os, and  Spencer in \cite{ramseycopies}.

 Let $t$ be a positive integer and $F$ be a graph.  Recall that $tF$ denotes  a graph that is a pairwise vertex-disjoint union of  $t$   copies of $F$.  Consider graphs $G$ and $H$ such that $F \a ({ G},{ H})$. Consider $s+t-1$ vertex disjoint copies of $F$ and color the edges of the resulting graph in red and blue. Then each copy of $F$ will either have a red induced copy of $G$ or a blue induced copy of $H$. The pigeonhole principle implies that there will be either a red copy of $sG$ or a blue copy of $tH$. This gives the following.

 \begin{observation}\label{main}
Let $G$, $H$ be graphs.
If $F \a ({ G},{ H})$ then
$$(s+t-1)F \a ({ sG},{ tH}).$$ Thus 
 \begin{equation}\label{maineq}
\IR(sG,tH)\ \le\ (s+t-1)\IR(G,H).
\end{equation}
\end{observation}

For all known so far results on induced Ramsey numbers for multiple copies, the inequality above holds as equality.
We raised the following question:  Does  there exist a pair  of graphs $(G,H)$ and a graph $X$ such that  $$|V(X)|\ { <}\ (s+t-1)\IR(G,H)~~~{\rm  and } ~~~
X \a (sG,tH)?$$

We answer this question in a positive by showing  that there are infinite classes of graphs where the inequality above is  strict and moreover, $IR(sG, tH)$ is  arbitrarily smaller than 
$(s+t-1)\IR(G,H)$.  On the other hand, we provide further examples of classes of graphs for which the inequality above holds as equality.\\

 Already in the case when $H=2K_2$, we observe a different behavior of $\IR(G, H)$ depending on $G$. 
 
 \setcounter{theorem}{5}
 
 \begin{theorem}
Let $G$ be a connected graph, $s$ be an integer, $s\geq|V(G)|$. Then 
$\IR(sG, 2K_2) = (s+1) \IR(G, K_2) = (s+1) |V(G)|$. 
 \end{theorem}
I.e., in this case the inequality (\ref{maineq}) holds as an equality. Let $P_n$ denote a path on $n$ vertices.
 \setcounter{theorem}{7}
\begin{theorem} $(1)$ For an integer $n\geq 5$, 
$\IR(P_n, 2K_2) = n+2 < 2n = 2\IR(P_n, K_2)$.
\end{theorem}
I.e., in this case  the inequality (\ref{maineq})  is strict and provides an arbitrarily large gap between $\IR(sG,tH)$ and $(s+t-1)\IR(G,H)$.
We prove these theorems and more results on  paths and matchings in Section \ref{Matchings}.

We find bounds on induced Ramsey numbers for short paths and complete graphs in Section \ref{Paths-cliques}.
Finally, we prove the following result in Section \ref{Triangles}.
 \setcounter{theorem}{11}
\begin{theorem}
Let $k$ be an integer,  $k\geq 2$. Then $\IR(K_3, kK_3)= 6k$.
\end{theorem}

We give known results on induced Ramsey numbers of multiple copies of graphs in Section \ref{Known}. 
In the last section we state some general observations.

~\\ 
 
 \section{Known results on induced Ramsey numbers for multiple copies} \label{Known}
 
\setcounter{theorem}{0} 
 
 First we introduce some basic notation.
 For a graph $F$ and subsets of vertices $S$ and  $S'$,  $F[S]$ denotes a graph induced by  $S$,  
$F[S, S']$ denotes a bipartite subgraph of $F$ containing all edges with one endpoint in $S$ and another in $S'$,  $F[\{x\}, S]$ is denoted $F[x,S]$. We denote the vertex and the edge sets of $F$ by $V(F)$ and $E(F)$, respectively.

For graphs $G$, $H$,  the  vertex-disjoint union  of $G$ and $H$ is denoted $G\cup H$,   $G\setminus H$ denotes a graph obtained from 
$G$ by removing $V(H)$. The  independence number of a graph $G$, i.e. the size of the largest set of mutually nonadjacent vertices, is denoted   $\alpha(G)$. The symbols $P_n$,  $K_n$, and $S_n$ stand for a path on $n$ vertices,  a complete graph on $n$ vertices, and a star with $n$ edges. For all other graph theoretic notions we refer the reader to the books of West \cite{west} and Diestel \cite{diestel}.

For any $n$-vertex graph $G$,  $\IR(K_2,G)=n$. Gorgol and {\L}uczak \cite{gorgolluczak} obtained the exact value of induced Ramsey number for a matching versus a complete graph and Gr\"unewald for two matchings \cite{grunewald}.

\begin{theorem}\cite{gorgolluczak}
Let $s\ge 1$ and $n\ge 2$ be integers. Then $\IR(sK_2,K_n)\ =\ sn$.
\end{theorem}

\begin{theorem}\cite{grunewald}
Let $s, t\ge 1$  be integers. Then $\IR(sK_2, tK_2)=2(s+t-1)$.
\end{theorem}

Kostochka and Sheikh \cite{kostoczka} considered the case when one graph in a pair is a $P_3$.

\begin{theorem}\cite{kostoczka} 
For any positive integers $n_1, \ldots, n_m$, 
$$\IR(P_{3},\bigcup_{i=1}^m K_{n_i})\ = \sum_{i=1}^m \binom{n_i+1}{2}\ =\ \sum_{i=1}^m \IR(P_3,K_{n_i}).$$
\end{theorem}

\begin{corollary}\label{p3tkn}
For any positive integer $t$, $\IR(P_{3},tK_{n})\ =\ t\IR(P_3,K_{n})=t\left(\binom{n}{2}+n\right)$.
\end{corollary}

\begin{theorem}\cite{kostoczka}\label{P3-multipartite}
Let  $H_i$, $i=1,2,\dots,m$ be  a complete multipartite graph. Then $\IR(P_{3},\bigcup_{i=1}^m H_{i})\ =\sum_{i=1}^m \IR(P_3,H_i).$
\end{theorem}

\begin{corollary}
Let  $H$ be a complete multipartite graph and $t$ be a positive integer.  Then $\IR(P_{3},tH)\ =\ t\IR(P_3,H)$.
\end{corollary}

\begin{theorem} \cite{kostoczka}
For any positive integer $s$, $7s\geq \IR(P_{3},sP_{4})\ \ge 6.1s$.
\end{theorem}

%%%%%%%%%%%%%%%%%%%%%%%%%%%%%%%%%%%%%%%%%%%%
%%%%%%%%%%%%%%%%%%%%%%%%%%%%%%%%%%%%%%%%%%%%
%%
%%.    MATCHINGS
%%
%%%%%%%%%%%%%%%%%%%%%%%%%%%%%%%%%%%%%%%%%%%%
%%%%%%%%%%%%%%%%%%%%%%%%%%%%%%%%%%%%%%%%%%%%

\section{Induced Ramsey numbers $G$ versus $2K_2$}\label{Matchings}

Obviously  $\IR(G,K_2)=|V(G)|$. We consider  $\IR(sG,2K_2)$. The next theorem shows that if $s$ is large enough, the equality in \eqref{maineq} holds.

\begin{theorem}\label{g2k2}
Let $G$ be a connected graph and $s\ge |V(G)|$. Then $\IR(sG,2K_2)=(s+1)|V(G)|= (s+1)\IR(G, K_2)$.
\end{theorem}

\begin{proof}
The inequality $\IR(sG,2K_2)\leq (s+1)|V(G)|$  follows from   \eqref{maineq}.  We shall show next that  $\IR(sG,2K_2)\geq (s+1)|V(G)|$ by proving that any graph on  $(s+1)|V(G)|-1$
vertices  can be edge-colored red and blue such that there is  neither red induced copy of  $sG$ nor blue induced copy of $2K_2$. Let $F$ be an arbitrary graph on $(s+1)|V(G)|-1$ vertices. 
We say that an induced subgraph of $F$ isomorphic to $sG$ is a {\it bundle}. 
Assume that any red-blue edge coloring of $F$ contains either a red bundle of a blue $2K_2$ as an induced subgraph.\\

We see that $F$ contains at least one bundle, otherwise we can color all edges of $F$ red.
Let $G_1, \ldots, G_s$ be copies of $G$ forming a bundle, with respective vertex sets $X_1, \ldots, X_s$, let $Y$ be the set of remaining vertices, i.e., $Y = V(F)-V(G_1\cup \cdots \cup G_s)$.  For any other bundle in $F$, we say that this bundle intersects $X_i$ nontrivially if it does not contain $X_i$ in is vertex set and thus does not contain respective induced copy of $G$. Note that each bundle contains at least one vertex from each $X_i$ because otherwise the total number of vertices in the bundle is at most $s|G|+s-1 -|G|< s|G|$.
Note that $|Y|= |V(G)|-1\leq s-1$.\\

For a fixed $i\in \{1, \ldots, s\}$, color an edge of $G_i$ blue and color all other edges of $F$ red. Then we see that there is a red bundle $H_i$. This bundle 
intersects $X_i$ nontrivially. Thus the bundle contains a copy of $G$ with vertices in $X_i$ and $Y$.  Since $G$ is connected,  there is an edge between $X_i$ and $Y$ for each 
$i=1, \ldots, s$.\\

Let $Q$ be an auxiliary bipartite graph with one part $Y$ and another  $X=\{X_1, \ldots, X_s\}$  and 
$X_iy\in E(Q)$ iff there is a bundle containing a copy of $G$ with an edge between $y$ and $X_i$. By the previous remark $|N_Q(X_i)|\geq 1$.
Consider a smallest subset $X'$ of $X$ such that $|N_Q(X')|<|X'|$.   Note that $X'$ is well defined since $N_Q(X)\subseteq Y$ and $|Y|<|X|$. We see that $|X'|>1$.  \\

Consider a bundle $H$ intersecting  the largest number $t$ of $X_i$'s nontrivially,  for $X_i\in X'$.
Let $X'= X''\cup X'''$, where for each $X_i\in X''$,  $H$ intersects $X_i$ nontrivially, and for each $X_i\in X'''$, $X_i\subseteq V(H)$, i.e.  $G_i\subseteq H$.
Let $H''$ be a union of copies of $G$ from $H$ that intersect members of  $X''$. Observe that $H''=t'G$ for $t'\geq t=|X''|$. Indeed, otherwise 
 $$|\cup_{X_i\in X''} X_i - V(H)|\geq  t|G| - t'(|G| -1)\geq t|G|-(t-1)|G| \geq |G|.$$ Thus the number of vertices of $F$ not in $H$ is at least  $|G|$, a contradiction.
 Since each copy of $G$ in $H''$ has a vertex in $Y$,  we see that $|V(H'')\cap Y| \geq |X''|$.
 Consider $X'''$.  Since there are no edges of $F$ between $\cup_{X_i\in X'''} X_i$ and $V(H)\cap Y$ and there are no edges of $Q$ between $X'$ and $Y - N(X')$, 
 $N_Q(X''') \subseteq N(X') - (V(H'')\cap Y)$. Thus $|N_Q(X''')|\leq |N(X')| - |X''| < |X'| - |X''|= |X'''|$, a contradiction to minimality of $X'$.
\end{proof}

In the next theorem we show  the lower bound on the induced Ramsey number for a pair $(G,2K_2)$.
\begin{theorem}\label{g2k2lower}
Let $G$ be a  graph without isolated vertices. Then $$\IR(G,2K_2)\ge |V(G)|+2.$$
\end{theorem}

\begin{proof}
Let $G$ have $n$ vertices. Consider an arbitrary graph $F$ on $n+1$ vertices. 
We shall show that $F$ can be edge-colored so that there is no induced red $G$ and no induced blue $2K_2$ in $F$.
If there is no induced $G$ in $F$, color all edges of $F$ red. So assume there is an induced copy $G'$  of  $G$. Consider the vertex $v$ of $F$ not in $G'$.
Let $u$ be a vertex of $G'$ incident to $v$ if such exists, let $u$ be an arbitrary vertex of $G'$ otherwise.
Color all edges incident to $v$ blue and  color one edge of $G'$ incident to $u$  blue; color all other edges red. 
This is a desired coloring, so $\IR(G, 2K_2)>n+1$.
\end{proof}

\begin{corollary}
Let $G$ be any graph on $n$ vertices and no isolated vertices.
Then $n+2 \leq  \IR(G, 2K_2) \leq 2n.$ Moreover both bounds can be  attained.
\end{corollary}

\begin{proof}
Theorem \ref{g2k2lower} gives the lower bound and Observation \ref{main} implies $\IR(G, 2K_2)\leq 2 \IR(G, K_2)=2n$, giving an upper bound.  Since \linebreak $\IR(K_2, 2K_2)=4$, both bounds are tight.
We can see that the bounds are tight for some graphs $G$ with  arbitrarily many vertices since \linebreak
$\IR(P_n, 2K_2)=n+2$, for $n\geq 5$, as we shall see in the next section and 
$\IR(K_n, 2K_2)= 2n$ by \cite{gorgolluczak}.
\end{proof}

In Theorem \ref{paths} we show that this lower bound is sharp for instance when $G$ is a path on at least five vertices.

%%%%%%%%%%%%%%%%%%%%%%%%%%%%%%%%%%%%%%%%%%%%%%%%%%%%%%%%%%%%%%%%%
%%%%%%%%%%%%%%%%%%%%%%%%%%%%%%%%%%%%%%%%%%%%%%%%%%%%%%%%%%%%%%%%%
%%
%%              PATHS
%%
%%%%%%%%%%%%%%%%%%%%%%%%%%%%%%%%%%%%%%%%%%%%%%%%%%%%%%%%%%%%%%%%%
%%%%%%%%%%%%%%%%%%%%%%%%%%%%%%%%%%%%%%%%%%%%%%%%%%%%%%%%%%%%%%%%%

~\\
\section{Paths and matchings}  \label{path}

Let for $a$, $b$ positive integers the symbol $\rem(a,b)$ denotes the reminder of a division $a$ by $b$.
~
\begin{theorem}\label{paths}
Let $n$, $s$, $t$ be positive integers. Then

\begin{enumerate}
\item{}
 $\IR(P_n,2K_2)=\begin{cases}
                           n+3, \ n=3,4,\\ 
                          n+2, \ n\ge 5,
                  \end{cases}$
 \item{}                 
 $\IR(sP_n,2K_2)\leq  sn+s+1, \ 2\le s\leq n-1$, $n\ge 4$,
 
  \item{}                 
 $\IR(sP_n,2K_2) = sn+s+1, \  s = 2,3$, $n\ge 4$,

\item{}                          
$\IR(sP_n,2K_2)=(s+1)n, \ s\ge n$, $n\ge 4$,

\item{}
$\IR(sP_3, 2K_2) = 3s+3 = (s+1) \IR(P_3, K_2)$, $s\ge 1$.
                 
\item{}
$\IR(P_3,tK_2)=3t$,   $t\ge 1$,

\item{}
$3t+1 \leq \IR(P_4,tK_2)\leq 7\lfloor t/2\rfloor +4\cdot \rem(t,2)$, $t\geq 1$,

\item{}
 $\IR(P_n,tK_2)\le \lceil t/2\rceil n +t -\rem(t,2)$,  $n\ge 5$,  $t\ge 1$,

\item{}
$\IR(P_{3},tP_{3})\ =\ 4t$,  $t\ge 1$.

\end{enumerate}
\end{theorem}

\begin{proof}

(1)  To show that $\IR(P_n, 2K_2) \leq n+2$ for $n\geq 5$, 
we shall prove that $C_{n+2}\a (P_n, 2K_2)$. Consider an edge-coloring of $C$ in red and blue. 
If there is a red $P_n$, we are done. Assume that there are no red $P_n$'s.
Since deleting any two consecutive vertices in $C$ leaves $P_n$ and it is not red, we see using that fact that $n+2\geq 7$
that there are  two blue edges at distance at least two in $C$. These edges form induced $2K_2$.
Thus $\IR(P_n, 2K_2)\leq n+2$.   The lower bound comes directly from Theorem \ref{g2k2lower}.

 We know that $\IR(P_3,2K_2)=6$ (cf. Corollary \ref{p3tkn}).
 In turn for $(P_4,2K_2)$ the argument above shows that $C_{7}\a (P_4,2K_2)$. Consider a graph $F$ on $6$ vertices. We shall show that it can be colored with no  red induced $P_4$ and no blue induced $2K_2$. We can assume that there is an induced copy $P$ of $P_4$ in $F$, otherwise we can color all edges red. 
 Let $ P=(x_1, x_2, x_3, x_4)$ and  let $x, y$ be the vertices of $F$ not in $P$.

If $xy\in E(F)$,   color $xy$, $x_1x_2$, and $x_3x_4$ red and all remaining edges  blue. Then there is no red $P_3$ and any blue edge is incident to $x_2$ or $x_3$ that are adjacent, thus there is no induced blue $2K_2$.

If $xy \not\in E(F)$ and $x,y, x_1,x_4$ induce a $P_4$, say with $x_4$ being an endpoint of this $P_4$,  color all edges  incident to $x_3$ and all edges incident to $x_4$ blue, and the remaining edges red. Then the red edges form a star, thus there is no red  $P_4$, each of the blue edges is incident to one of two adjacent vertices, thus there is no blue induced $2K_2$. 

If $xy \not\in E(F)$ and $x,y, x_1,x_4$ does not induce a $P_4$, color all edges incident to $x_2$ and all edges incident to $x_3$ blue, and all other edges red. As before we see that there is no induced blue $2K_2$. The red edges are spanned by 
$x,y, x_1, x_4$, that does not induce a $P_4$.

Thus $\IR(P_4, 2K_2)>6$. 
Together with the upper bound, we get  $\IR(P_4,2K_2)=7$.\\

(2)
Let  $ \ 2\le s \le n-1$.  Note that $sP_n\prec P_{sn+s-1}$. To see this, delete every $(n+1)^{\rm st}$ vertex from $P_{sn+s-1}$.
Thus $\IR(sP_n,2K_2)\leq \IR(P_{sn+s-1}, 2K_2)$. By item (1) we have that $\IR(P_{sn+s-1}, 2K_2)\leq sn+s+1$.
Therefore $\IR(sP_n,2K_2)\leq sn+s+1$.\\

(3) The upper bound follows directly from item (2). 

As for the lower bound consider a graph $F$ on $ns +s $ vertices. It contains a bundle $B$ (induced copy
of $sP_n$). Let $Y$ be the set of all vertices of $F$ not in $B$, i.e., $|Y|=s$.

Fix any component $P$ in $B$ and color any three consecutive edges on it blue and all the remaining edges of $F$ red.
Then we see that there must be a red bundle, otherwise we are done. Let this bundle
be $B_1$. We see that $B_1$ can use at most $n-2$ vertices from $P$, so it uses at least two vertices of $Y$. 

Assume that there is exactly one vertex of $Y$ adjacent to $P$. 
If $|V(P)-V(B_1)| = 3$, then $s = 3$ and $Y$ is contained in $V(B_1)$. Moreover, two other paths of $B$ are components of $B_1$, so they have no neighbors in $Y$, a contradiction.
If $|V(P)-V(B_1)| = 2$, then the remaining $n-2$ vertices of $P$ together with the only one vertex from $Y$ can induce a path on at most $n-1$ vertices, so it could not be a component of $B_1$.

So we can assume that there are at least two vertices in $Y$ sending edges to  $P$.  If $B_1$ uses
at most $n-1$ vertices from each of the remaining $s-1$ components of $B$, we see
that  $B_1$ omits at least $2+ (s-1)=s+1$ vertices of $F$, i.e., it contains at most
$ns-1$ vertices, a contradiction.  Thus, there is a component $P'$  of $B$   so that all its vertices are contained in
$B_1$. Since $B_1$ is an induced subgraph of $F$,  $P'$ is a component of $B_1$ as
well. Since $P$ could be chosen to be an arbitrary component of $B$, let $P=P'$.
We see that on one hand $P'$ sends edges to at least two vertices of $Y$, on the other hand, it does not send edges to $Y\cap B_1$.
Since $|Y\cap B_1|\geq 2$, there are at least two vertices of $Y$ that send edges to $P'$ and at least two vertices of $Y$ that do not send edges to $P'$. So $|Y|\geq 4$, a contradiction to the fact that $|Y|=s\in \{2,3\}$.\\

(4) By Theorem \ref{g2k2} $\IR(sP_n,2K_2)=(s+1)n$ for $s\ge n$. \\

(5) The upper bound $\IR(sP_3, 2K_2)\leq 3s+3$ follows from Observation \ref{main}.
For the lower bound, consider a graph $F$ on $3s+2$ vertices. We shall show that $F$ can be edge-colored so that there is no induced red $sP_3$ and no induced blue $2K_2$. We can assume that $sP_3 \prec F$ otherwise  we can color all edges of $F$ red. Let $a_i$, $b_i$, $c_i$ be the vertices of $i$-th path $P_3$ and $x$ and $y$ be the remaining vertices of $F$.

Assume first that  for some $i\in [s]$, there is an edge between  $\{a_i, c_i\}$ and $\{x, y\}$, assume w.l.o.g., that $a_1x\in E(F)$. Color all edges incident to $a_1$ and $b_1$ blue, the rest of the edges red. Then there is no blue $2K_2$ and we must have a red copy $B$ of $sP_3$. We see that $B$ could contain at most $2$ vertices from $\{x,a_1, b_1, c_1\}$ otherwise it would induce a blue edge. Since $|V(B)|= |V(F)|-2$,  $B$ contains exactly two vertices from $\{x,a_1, b_1, c_1\}$ and thus contains $y$ and 
all paths $(a_i, b_i, c_i)$, $i=2, \ldots, s$. Thus there is a red $P_3$ induced by $\{x,y, a_1, b_1, c_1\}$ and containing $y$.
Since all edges incident to $a_1$ and $b_1$ are blue,  it must be induced by $\{x, y, c_1\}$.
Then we see that  $F[\{x,y, a_1, b_1, c_1\}]$ contains a $C_4$ with a pendant edge or a $C_5$ and thus does not induce $2K_2$. Color $F[\{x,y, a_1, b_1, c_1\}]$ blue and the rest red, it results in a desired coloring.

Now, we can assume that $\{x, y\}$ sends edges only to vertices in $\{b_1, \ldots, b_s\}$. Assume w.l.o.g.,  that $b_1x\in E(F)$. Color all edges incident to $x$ and to $b_1$ blue, the rest red.
This is a desired coloring.\\

(6) According to Corollary \ref{p3tkn} $\IR(P_3,tK_2)=3t$.\\

(7) By item $(1)$ we have that $C_7\a (P_4,2K_2)$.

 Hence $ (t/2)  C_7\a (P_4, tK_2)$ for $t$ even and $\lfloor t/2\rfloor C_7\cup P_4\a (P_4, tK_2)$ for $t$ odd. This gives the upper bound.

As for the lower bound let $F$ be an arbitrary graph on $3t$ vertices. We shall prove that there is an edge coloring of $F$ in two colors with no red induced $P_4$ and no blue induced $tK_2$. First observe that $tK_2 \prec F$ otherwise we could color all edges of $F$ blue. Let $x_iy_i$, $i=1,2,\dots,t$, be the edges of this induced matching and $z_i$, $i=1,2,\dots,t$ be the remaining vertices of $F$. Color all edges $x_iy_i$ red. We can assume that there is another induced blue $tK_2$ otherwise we color the remaining edges blue and obtain the desired coloring. As we can take exactly one vertex from each $x_iy_i$ to construct this new matching $M$, vertices $z_i$ form an independent set and all are involved in $M$. Without loss of generality we can assume that $x_iz_i$, $i=1,2,\dots,t$, are the edges of $M$. We color all of them red. So again we can assume that there is one more induced matching $M_1=tK_2$. The only possibility is that each $z_i$, $i=1,2,\dots,t$ has exactly one neighbor in $\{y_1,y_2,\dots,y_t\}$. Therefore $F$ consists of induced cycles of length divisible by $3$, i.e. $F=\bigcup_k C_{3t_k}$. We shall show that induced $C=C_{3\tau}$ can be colored without red induced $P_3$ and  blue induced $\tau K_2$.   If $\tau=1$, color $C= C_3$ red. If $\tau>1$,  color $C=C_{3\tau}$ so that the red subgraph  forms matching on all but at most one vertex of $C$. Then there is no red $P_3$ and the blue subgraph forms a disjoint union of edges and perhaps one $P_3$. Since consecutive blue edges on $C$ do not form an induced $2K_2$, the largest induced blue matching has at most $\lfloor \frac{3\tau + 1}{4} \rfloor< \tau$ edges. 

Let $M_k$ be the largest  induced blue matching in $C_{3t_k}$. Since $F=\bigcup_k C_{3t_k}$, the largest induced blue matching in $F$ has the cardinality
$\sum_k |M_k| < \sum_k t_k = t$, which completes the proof of the lower bound.\\

(8)  From item $(1)$ we have that $C_{n+2}\a (P_n,2K_2)$, $n\ge 5$. 
Hence $ (t/2)  C_{n+2}\a (P_n, tK_2)$ for $t$ even and $\lfloor t/2\rfloor C_{n+2}\cup P_n\a (P_n, tK_2)$ for $t$ odd. This gives the upper bound.\\

(9) This follows immediately from Theorem \ref{P3-multipartite} since $P_3=K_{2,1}$.

\end{proof}

%%%%%%%%%%%%%%%%%%%%%%%%%%%%%%%%%%%%%%%%%%%%
%%%%%%%%%%%%%%%%%%%%%%%%%%%%%%%%%%%%%%%%%%%%
%%
%%        PATHS versus CLIQUES
%%
%%%%%%%%%%%%%%%%%%%%%%%%%%%%%%%%%%%%%%%%%%%%
%%%%%%%%%%%%%%%%%%%%%%%%%%%%%%%%%%%%%%%%%%%%   
%~\\
\section{Short paths and complete graphs}\label{Paths-cliques}

As we mentioned  Kostochka and Sheikh showed that 
$\IR(P_{3},sK_{n})=s\IR(P_3,K_{n})=s\binom{n}{2}+sn$. We consider the case when there are multiple copies of  $P_3$ and one copy of $K_n$ instead.

\begin{theorem}\label{pathscomplete}
Let $s\ge 1$ and $n\ge 3$. Then 
$$\binom{n+1}{2} + (2s-2)(n-1) \leq \IR(sP_{3},K_{n})\leq s \IR(P_3, K_n) = s\binom{n}{2}+sn.$$
\end{theorem}

\begin{proof}
The upper bound follows from Observation \ref{main}.

For the lower bound, consider a graph $F$ on the smallest number of vertices such that $F\a (sP_3, K_n)$.
We see that there is a copy  of $K_n$ in $F$, otherwise we can color all edges of $F$ blue.

Let us denote this clique $K^0$ and colour it red. We see 
that $F\setminus V(K^0)$ contains a clique $K_{n-1}$ otherwise color $K^0$ red and 
the remaining edges of $F$ blue.  Denote this clique $K^1$ and colour $F_1=F[V(K^0)\cup V(K^1)]$
 red.  For $s\geq 2$, $F_1$ does not contain an induced copy of $sP_3$, so there is no red $sP_3$.  Similarly $F\setminus V(F_1)$
 contains a clique $K_{n-1}$ which we denote $K^2$ .
Repeating the above consideration we conclude
that apart from $K^0$ the graph $F$ contains $2s-1 $ pairwise vertex disjoint cliques $K_{n-1}$ denoted
by $K^1$, $K^2$, $\dots$, $K^{2s-2}$.
Let  $F'=F[V(\bigcup_{i=0}^{2s-2} K^i)]$. Color all edges of $F'$ red and color the edges of $K^{2s-1}$ red. Color the remaining edges blue. 
We see that there is no red $sP_3$, so there must be a blue $K_n$. Thus there is a copy of $K_{n-2}$ induced by the vertices of $F$ not in $\bigcup_{i=0}^{2s-1} K^i$.
Similarly, we observe that $F$ contains a vertex-disjoint union of $F'$ and a graph $K$ that is a vertex disjoint union of copies of  $K_{n-1}, K_{n-2}, \ldots, K_2$. 
If $V(F) = V(F') \cup V(K)$, we color all edges $F'$ red, all edges of $K$ red and the remaining vertices blue. Note that the blue color class forms an $(n-1)$-partite graph and thus does not contain $K_n$.  So, $F\na (sP_3, K_n)$, a contradiction. 
Therefore, 
$|V(F)|> |V(F')|+|V(K)| = n + (2s-2)(n-1) + (n-1) + (n-2) + \ldots + 2 $. 
In particular $|F| \geq \binom{n+1}{2} + (2s-2)(n-1)$. 
\end{proof}

A similar argument works for $(sG, K_n)$ with $G$ being a triangle free graph.

We can improve the lower bound on $\IR(2P_{3},K_{3})$ from $10$ as given in Theorem \ref{pathscomplete} to $11$.

\begin{theorem}
Then $\IR(2P_{3},K_{3})\ \ge\ 11$.
\end{theorem}

\begin{proof}
Let $F$ be an arbitrary graph on $10$ vertices. We shall prove that there is an edge-coloring of $F$ with no red copy of $2P_3$ and no blue copy of $K_3$, i.e. that $F\na (2P_3, K_3)$. \\

We can assume that  $F$ contains a vertex disjoint union of $K_3$ and $2K_2$. 
Indeed, $K_3$ exists otherwise we can color all edges of $F$ blue and $F\na (2P_3, K_3)$.
Coloring the edges of  a copy $K$ of $K_3$ red and all others blue implies that there must be a blue $K_3$, so there must be a copy  $K'$ of $K_2$ vertex-disjoint from $K$. Finally, coloring the subgraph of $F$ induced by vertices of $K$ and $K'$ red, and other edges blue, shows that there is a blue $K_3$, i.e., in particular a $K_2$ vertex disjoint from $K\cup K'$.
So, we indeed can assume that $F$ contains a vertex-disjoint union of $K_3$ and $2K_2$.
Note that any graph containing $K_3\cup 2K_2$ as a spanning subgraph does not contain $2P_3$ as an induced subgraph.\\

{\bf Case 1.}  For some copy $K''$ of a vertex-disjoint union of $K_3$ and $2K_2$, $F - V(K'')$ is not isomorphic to $P_3$.
In this case, color the edges of $K''$ and $F-V(K'')$ red and the remaining edges blue. This results in no induced red $2P_3$ and no blue $K_3$, so $F\na (2P_3, K_3)$.\\

{\bf Case 2.}  For any copy $K''$ of  a vertex-disjoint union of $K_3$ and $2K_2$, $F - V(K'')$ is  isomorphic to $P_3$.
We have then that $F$ contains a spanning subgraph that is a union of $K$, $K'$, and $P$, where $P$ is a copy of an induced $P_3$,  $K$ is isomorphic to $K_3$, and $K'$ is isomorphic to $2K_2$. 
By taking an edge $e$ of $K'$, an edge  $e'$  of $P$, we see that the vertices of $F-V(K)$ not incident to $e$ or $e'$ induce a copy of $P_3$. Thus $F-V(K)$ contains a spanning subgraph that is a union of three copies of $P_3$ that share exactly one vertex that is an endpoint in each of these $P_3$'s. Then we see that $F-V(K)$ does not contain a copy of an induced $2P_3$. Color all edges of $K$ and all edges of $F-V(K)$ red and the remaining edges blue. There is no induced red $2P_3$ and no blue $K_3$, so $F\na (2P_3, K_3)$.

\end{proof}

\section{Triangles}\label{Triangles}

Ramsey numbers for multiple copies of graphs were considered by Burr, Erd\H os and  Spencer in \cite{ramseycopies}. Their  paper contains, among others, the following result.

\begin{theorem}\cite{ramseycopies}\label{ramseytriangles}
Let $t\ge s\ge 1$ and $t\ge 2$ be integers. Then $R(sK_{3},tK_{3})\ =\ 2s+3t$.
\end{theorem}
We prove the following.

\begin{theorem}
Let $t$ be a positive integer. Then $\IR(K_{3},tK_{3})=6t$.
\end{theorem}

\begin{proof}
The upper bound follows immediately from (\ref{maineq}):
$$\IR(K_3, tK_3)\leq t \IR(K_3,K_3)= t R(K_3, K_3) = 6t.$$

\noindent
For the lower bound, we need a statement on induced matchings.\\

\noindent
{\it Claim~~} If $G$ is any graph on $n$ vertices, then there is a partition $V(G)=V_1\cup V_2$ such that any induced matching $M$ in $G$ contains at most $n/3$ edges 
with both endpoints in $V_1$ or in $V_2$, i.e., $|E(M[V_1]\cup M[V_2])|\leq n/3$.

Assume not, consider a partition and an induced matching $M$ such that $M$ has more than $n/3$ edges with both endpoints in one part of the partition.
Let a new partition $V_1'$, $V_2'$ be built so that each edge of $M$ has one endpoint in $V_1'$ and another in $V_2'$, the rest of the vertices are assigned to $V_1'$ or $V_2'$ arbitrarily. Then we see that $V_i'$ has an independent set of size  greater than $n/3$. Then any  matching in $V_i'$ has strictly less than  $|V_i'|-n/3$ edges, $i=1,2$.  So, any induced matching of $G$ contains less than  $|V_1'|+|V_2'|-n/3-n/3 = n/3$ edges with both endpoint in the same part. 
This concludes the proof of Claim.\\

Let $F$ be  a graph, $F\rightarrow (K_3, tK_3)$. We shall show that $|V(F)|\geq 6k$.  We can assume $tK_{3} \prec F$ otherwise we could color all edges of $F$ blue. Let $a_i$, $b_i$, $c_i$, $i=1,2,\dots,t$ be the vertices of these triangles and $X$ be the set of remaining vertices. Let $X=X'\cup X''$ be a partition of $X$ such that any induced matching  of $F[X]$ has at most $|X|/3$ edges with both endpoints in $X'$ or in $X''$. Such a partition exists by Claim.
Color $a_i b_i$, $b_ic_i$,   $F[a_i, X']$, $F[c_i,X'']$,   and $F[X', X'']$ red, $i=1, \ldots, t$, and all remaining edges blue.
We see that there is no red triangle. 
Assume that there is a blue induced copy of $tK_3$, denote it $H$. 
Any blue triangle has at most one vertex in $\{a_i, b_i, c_i\}$ for any $i=1, \ldots, t$.
Thus $H[X]$ contains a blue induced matching on $t$ vertices. This matching could have its edges only with both endpoints in $X'$ or both endpoints in $X''$ since all edges between $X'$ and $X''$ are red. By the way we chose a partition $X', X''$, there are at most $|X|/3$ such edges. Thus $|X|/3 \geq t$, i.e., $|X|\geq 3t$.
This implies that $|V(F)|\geq 6t$.
\end{proof}

\section{Further observations}

While the structure of a graph $F$ such that $F\a (G, tH)$ and $|V(F)|=\IR(G, tH)  = tIR(G,H)$ is clear, as it is simply a vertex disjoint union of $t$ copies of $F'$ such that $F'\a (G,H)$, the structure of such graphs $F$ so that $|V(F)|< t\IR(G,H)$ is not so clear. We claim that such a graph must be connected.

\begin{remark} \label{arrowing-connected}
Let $G$, $H$ be arbitrary connected graphs and $t$ be a positive integer.
Let  for $i=1, \ldots, t$,  $f_i = \IR(G,iH)$ and $F_i$ be a graph of order $f_i$ such that $F_i \a (G,iH)$.

Assume that  $f_t< \min_{\sum t_i=t} \sum f_{t_i}$. Then $F_t$ is connected.
\end{remark}

\begin{proof}
Assume to the contrary  that $F_t$ consists of $m>1$ components $S_1$, $S_2$, $\dots$, $S_m$. 
For $j=1, \ldots, m$,  let  $t_j$ be the largest integer such that $S_j\a (G,t_jH)$. Obviously $1\le t_j\le t$. We have that  $|S_j|\ge f_{t_j}$, so $f_t \geq f_{t_1} + \cdots + f_{t_m}$. Moreover $F_t \a (G, (t_1+\cdots + t_m)H)$. Since $F_t$ is a graph of a smallest order such that $F_t \a (G, tH)$, 
we have that $t_1+\cdots + t_m =t$. But we know that $f_t< f_{t_1} + \cdots +f_{t_m}$ since $t_1+\cdots t_m=t$. 
A contradiction.
\end{proof}

\bibliographystyle{amsplain}

\begin{thebibliography}{1}

\bibitem{Albertson_subchromatic}
M. O. Albertson, R. E. Jamison, S. T. Hedetniemi and S. C. Locke, {\it The subchromatic number of a graph}, Discrete Math., \textbf{74} (1989), 33--49.

\bibitem{becksize1}
J.~Beck, \emph{On the size {R}amsey number of paths, trees and circuits {I}},
  J. Graph Theory \textbf{7} (1983), 115--129.
  
%\bibitem{generalized_Bollobas}
%B. Bollob\'as and D. B. West, A note on generalized chromatic number and generalized girth, \emph{Discrete Math.}, \textbf{213} (2000), 29--34.
  
%\bibitem{generalized_Broere}
%I. Broere, S. Dorfling and E. Jonck, Generalized Chromatic Numbers and Additive Hereditary Properties of Graphs, \emph{Discussiones Mathematicae Graph Theory}, \textbf{22} (2002),  259--270.
  
\bibitem{burr}
S. Burr,
\emph{Ramsey numbers involving graphs with long suspended paths}
J. Lond. Math. Soc., \textbf{24} (1981), pp. 405--413.

\bibitem{ramseycopies}
S. Burr, P. Erd\H os and J. Spencer,
\emph{Ramsey theorems for multiple copies of graphs}
Transactions of the American Mathematical Society, \textbf{209} (1975), pp. 87--99.

\bibitem{confoxsud}
 D. Conlon, J. Fox and B. Sudakov, \emph{On two problems in graph
Ramsey theory}, Combinatorica \textbf{32}(2012), 513–-535.
  
\bibitem{chvatalharary}
 V.  Chv\'atal  and  F.  Harary,  \emph{Generalized  Ramsey  theory  for  graphs.  III.  Small  off-diagonal
numbers}, Pacific J. Math. \textbf{41} (1972), 335–-345.
  
\bibitem{chvatal}
 {V. Chv\'atal}, \emph{Tree-complete graph {R}amsey number}, J. Graph Theory \textbf{1} (1977), 93.
 
  

\bibitem{deuber}
W. Deuber, \emph{A generalization of {R}amsey's theorem}, Infinite and finite
  sets (R.~Rado A.~Hajnal and V.~S\'os, eds.), vol.~10, North-Holland, 1975,
  pp.~323--332.

\bibitem{diestel}
R. Diestel, \emph{Graph Theory 5th ed}, Graduate Texts in Mathematics, Vol. 173, Springer-Verlag, Heidelberg, 2017.


\bibitem{erdramskn}
P. Erd\H os,
\emph{Some remarks on the theory of graphs},
Bull. Amer. Math. Soc., 53 (1947), pp. 292--294

\bibitem{erdirconj}
P. Erd\H os, \emph{On some problems in graph theory, combinatorial analysis and combinatorial
number theory.} Graph theory and combinatorics (Cambridge, 1983) (1984),
1–-17, Academic Press, London, 1984.

\bibitem{erdhapo}
P.~{Erd\H os}, A.~Hajnal and L.~P\'osa, \emph{Strong embeddings of graphs into
  colored graphs}, Infinite and finite sets (R.~Rado A.~Hajnal and V.~S\'os,
  eds.), vol.~10, North-Holland, 1975, pp.~585--595.
 
\bibitem{fox_sudakov_induced} 
  J. Fox and B. Sudakov, {\it Induced Ramsey-type theorems}, Advances in Mathematics, \textbf{219} (2008), 1771--1800.

\bibitem{gorgol}
I. Gorgol, \emph{A note on a triangle-free--complete graph induced Ramsey number},
Discrete Math. \textbf {235}, 1--3 (2001), 159–-163.

\bibitem{gorgol-LB}
I. Gorgol, \emph{A note on lower bounds for induced Ramsey numbers},
accepted to Discuss. Math. Graph Theory 

\bibitem{gorgolluczak}
I. Gorgol and T. {\L}uczak, \emph{On induced Ramsey numbers}, Discrete Math. \textbf{251},
1--3 (2002), 87–-96.

\bibitem{grunewald}
A. Gr\" unewald, \emph{Induced Ramsey Numbers of Graphs}, Bsc. thesis, Karlsruher Institut
f\"ur Technologie, 2016.

\bibitem{irn}
F. Harary, J. Ne\v{s}šet\v{r}il  and V. R\"odl, \emph{Generalized Ramsey theory for graphs.
XIV. Induced Ramsey numbers}, Graphs and other combinatorial topics (Prague,
1982) 59 (1983), 90–-100.


\bibitem{haxkohlu}
P.~Haxell, Y.~Kohayakawa and T.~{\L}uczak, \emph{The induced size-{R}amsey
  number of cycles}, Combinatorics, Probab. Comput. \textbf{4} (1995),
  217--240.

\bibitem{kohprro}
Y.~Kohayakawa, H. J. Pr\"omel and V.~R\"odl, \emph{Induced {R}amsey numbers},
  Combinatorica \textbf{18} (1998), no.~3, 373--404.
  
\bibitem{kostoczka}
A. Kostochka and N. Sheikh, \emph{On the induced Ramsey number IR(P3;H)},
Topics in discrete mathematics 26 (2006), 155–-167.

\bibitem{luczdeg}
T.~{\L}uczak and V.~R\"odl, \emph{On induced {R}amsey numbers for graphs with
  bounded maximum degree}, J. Combin. Theory, Ser. {\bf B} \textbf{66} (1996),
  324--333.

\bibitem{rodl}
V.~R\"odl, \emph{The dimention of a graph and generalized  {R}amsey theorems}, Master's thesis, Charles
  University, Prague, 1973.
  
  \bibitem{rodlp}
V.~R\"odl, \emph{A generalization of {R}amsey Theorem},  Graphs, Hypergraphs and Block Systems, Zielona G\'ora, 1976, 211--219.


\bibitem{schaefer}
M. Schaefer and P. Shah, \emph{Induced graph Ramsey theory}, Ars Combin. \textbf{66}
(2003), 3-–21.


\bibitem{west}
D. West, \emph{Introduction to Graph Theory}, Pearson, 2017.
\end{thebibliography}

\end{document}